\documentclass[11 pt,reqno]{amsart}

\usepackage{amsmath, amsfonts, amssymb, amscd, enumerate, color}
\theoremstyle{plain}
\newtheorem{theo}{Theorem}[section]

\newtheorem{prop}[theo]{Proposition}

\theoremstyle{definition}

\newtheorem{example}[theo]{Example}
\newtheorem{definition}[theo]{Definition}

\newenvironment{pf}{\noindent{\it Proof. }}{$\square$\par\medskip}
\newenvironment{pfns}{\noindent{\it Proof. }}{\par\medskip}

\theoremstyle{plain}
\newtheorem{lemma}[theo]{Lemma}
\newtheorem{theorem}[theo]{Theorem}

\newtheorem{proposition}[theo]{Proposition}

\theoremstyle{definition}

\newtheorem{remark}[theo]{Remark}

\renewcommand{\=}{:=}

\newcommand{\beq}{\begin{equation}}
\newcommand{\eeq}{\end{equation}}
\renewcommand{\a}{\alpha}
\renewcommand{\b}{\beta}

\renewcommand{\d}{\delta}

\renewcommand{\l}{\lambda}
\newcommand{\m}{\mu}
\renewcommand{\o}{\omega}

\newcommand{\D}{\Delta}

\newcommand{\bC}{\mathbb{C}}

\newcommand{\bR}{\mathbb{R}}

\newcommand{\bH}{\mathbb{H}}

\newcommand{\bN}{\mathbb{N}}

\newcommand{\bS}{\mathbb{S}}

\renewcommand{\gg}{\mathfrak{g}}

\newcommand{\gk}{\mathfrak{k}}

\newcommand{\gm}{\mathfrak{m}}

\newcommand{\gX}{\mathfrak{X}}

\newcommand{\su}{\mathfrak{su}}

\newcommand\SU{\mathrm{SU}}
\newcommand\U{\mathrm{U}}

\newcommand{\cA}{\mathcal{A}}
\newcommand{\cB}{\mathcal{B}}
\newcommand{\cC}{\mathcal{C}}
\newcommand{\cD}{\mathcal{D}}
\newcommand{\cE}{\mathcal{E}}
\newcommand{\cF}{\mathcal{F}}

\newcommand{\cL}{\mathcal{L}}
\newcommand{\cM}{\mathcal{M}}

\newcommand{\cO}{\mathcal{O}}
\newcommand{\cP}{\mathcal{P}}

\newcommand{\cR}{\mathcal{R}}
\newcommand{\cS}{\mathcal{S}}

\newcommand{\cU}{\mathcal{U}}
\newcommand{\cV}{\mathcal{V}}

\newcommand{\cZ}{\mathcal{Z}}

\renewcommand{\square}{\kern1pt\vbox
{\hrule height 0.6pt\hbox{\vrule width 0.6pt\hskip 3pt
\vbox{\vskip 6pt}\hskip 3pt\vrule width 0.6pt}\hrule height0.6pt}\kern1pt}

\DeclareMathOperator\Tr{Tr}

\DeclareMathOperator\vol{vol}
\DeclareMathOperator\Id{Id}

\renewcommand\Im{\operatorname{Im}}

\newcommand{\Hom}{{\operatorname{Hom}}}

\newcommand{\wt}{\widetilde}

\newcommand{\n}{\nabla}

\newcommand{\be}{\begin{equation}}
\newcommand{\ee}{\end{equation}}

\def\<#1,#2>{\langle\,#1,\,#2\,\rangle}
\newcommand{\arr}{\begin{array}{rlll}}
\newcommand{\ea}{\end{array}}
\newcommand{\bea}{\begin{eqnarray}}
\newcommand{\eea}{\end{eqnarray}}
\newcommand{\bean}{\begin{eqnarray*}}
\newcommand{\eean}{\end{eqnarray*}}

\def\sideremark#1{\ifvmode\leavevmode\fi\vadjust{%            The remark
\vbox to0pt{\hbox to 0pt{\hskip\hsize\hskip1em%               will appear only
\vbox{\hsize3cm\tiny\raggedright\pretolerance10000%          on the side
\noindent #1\hfill}\hss}\vbox to8pt{\vfil}\vss}}}%           in 3cm
\newcounter{ssig}
\newcounter{ttig}
\newcommand{\Ric}{\operatorname{Ric}}

\newcommand{\Hess}{\operatorname{Hess}}

\newcommand{\R}{\operatorname{R}}
\newcommand{\DD}{\operatorname{D}\!}
\newcommand{\Hs}{\operatorname{H}^s}
\newcommand{\sob}[1]{\operatorname{H}^{#1}}
\newcommand{\Sol}{\cS ol}
\newcommand{\Iso}{\operatorname{Iso}}

\title[On  moduli spaces of Ricci solitons] {On moduli spaces of Ricci solitons}

\author{Fabio Podest\`a and Andrea Spiro}


\begin{document}
\begin{abstract} We study deformations of shrinking Ricci solitons on a compact manifold $M$, generalising the classical theory of deformations of Einstein metrics. Using appropriate notions of twisted slices $S_f$ inside the space of all Riemannian metrics on $M$, we define the infinitesimal solitonic deformations and the local solitonic pre-moduli spaces. We prove the existence of a finite dimensional submanifold of $S_f\times \cC^\infty(M)$, which contains the pre-moduli space of solitons around a fixed shrinking Ricci soliton as an analytic subset. We define solitonic rigidity and give criteria which imply it.
\end{abstract}

\maketitle
\section{Introduction}
In this paper we study Ricci solitons on  manifolds  setting up the theory of their deformations. We recall that a Ricci soliton on a manifold $M$ is a complete Riemannian metric $g$, with  Ricci tensor   satisfying the equation  $\Ric + \frac 12 \cL_Xg - cg =0$ for some complete vector field $X$ and a constant $c$. Solitons are remarkably  interesting metrics as they evolve in a particularly easy way under the Ricci flow, namely by diffeomorphisms and homotheties, and they  appear as natural generalisations of Einstein metrics,  which are the  Ricci solitons  with  $X = 0$. The reader is referred to \cite{Ca1, Ch, ELM} for a detailed exposition of the theory of Ricci solitons.\par
\smallskip
It is known that on a compact manifold $M$ every Ricci soliton is of gradient type, i.e.,  with  associated vector field equal to   the gradient $X = \n f$ of a smooth function $f$. It is then natural to consider the space
$$\cP =\left\{ (g,f)\in \cM \times \cC^\infty(M)\ |\ \int_Me^{-f}\mu_g = (2\pi)^{\frac n2}\right\}\ ,$$
 where $\cM$ is the space of all Riemannian metrics on the $n$-dimensional manifold $M$,  and to define the set of  {\it normalised Ricci solitons} $\Sol$ as the set of all pairs $(g,f)\in \cP$  satisfying the normalised equation
$$\Ric + \Hess(f) - g = 0\ .$$
We remark that   $\Sol$ coincides with the set of all critical points in $\cP$ of Perelman's entropy functional $W(g,f)$ or, equivalently,  with  the zero set of the corresponding Euler-Lagrangian operator $\bS\! =\! (\bS_1, \bS_2)\!:\!\cP \longrightarrow \cC^\infty(S^2(T^*M))\times \cC^\infty(M)$
 (see \eqref{ELequations}).
 If we compare all this to the classical theory of Einstein metrics (which can be considered as the critical  points  of  the total scalar curvature $H(g)$, in the space $\cM_1$ of Riemannian metrics of unitary volume) we have that   Perelman's functional $W(g,f)$ and  the equations $\bS_1(g, f) = \bS_2(g,f) = 0$ are non trivial analogues of $H(g)$ and  Einstein equation, respectively, that are well
 suited to our purposes. \par
 \smallskip
The aim of this paper is to study the moduli space $\Sol/\cD$ of  solitons w.r.t. the natural action of the set  $\cD$  of all diffeomorphisms of $M$, around a fixed normalised Ricci soliton $(g_o,f_o)$. In order to do this, we first define and prove the existence of  modified Ebin slices, called {\it  $f$-twisted slices $\cS_f$}, for the $\cD$-action on the space $\cM$. We then define the {\it solitonic  pre-moduli space} at $(g_o,f_o)$ as the intersection of $\Sol$ with the set $\cS_{f_o}\times \cC^\infty(M)$. Such solitonic pre-moduli space is obviously acted on by the isometry group $\mbox{Iso}(g_o)$ and the orbit space of this action is a local model
for the moduli space $\Sol/\cD$ around the class of $(g_o,f_o)$.\par
\smallskip
The space of {\it infinitesimal solitonic deformations} is then defined as the linear subspace of the tangent space
$T_{(g_o,f_o)}\cP$
  given by the kernel of the differential of $\bS$ at $(g_o,f_o)$. An element $(h,a)\in T_{(g_o,f_o)}\cP\simeq \cC^\infty(S^2(T^* M)) \times \cC^\infty(M)\ $ in such a linear subspace is called {\it essential} if the tensor field $h$
  belongs to the tangent space of the $f_o$-twisted slice $S_{f_o}$ at $g_o$. Our first main result states that {\it the space of all essential infinitesimal solitonic deformations is finite dimensional}.\par
  \smallskip

In order to provide a description of the solitonic pre-moduli space,   we work in the Hilbert manifold $\cP^s$ ($s\geq [n/2]+3$) given by
$$\cP^s =\left\{\ (g,f)\in \cM^s \times H^s(M)\ |\ \int_Me^{-f}\mu_g = (2\pi)^{\frac n2}\ \right\}\ ,$$ where $H^s(M)$ is the Sobolev space of order $s$ and $\cM^s$ denotes the space of all Riemannian metrics which are $H^s$-sections of the bundle $S^2(T^*M)$. This setting is convenient because it allows  the use of the Implicit Function Theorem and brings to our main result, {\it which consists in  a local description of the pre-moduli space as an analytic subset of a finite dimensional analytic submanifold $\cZ^s \subset \cP^s$  through $(g_o,f_o)$,  whose tangent space is given by  the set of essential infinitesimal solitonic deformations}.
\par
\smallskip
We then introduce the notion of ({\it weak) solitonic rigidity} and  provide some criteria that imply it, in particular when the fixed metric $g_o$ is Einstein. We analyse the case of compact symmetric spaces of rank one, showing that  the complex projective spaces $\bC P^n$ are the only ones with non trivial spaces of essential infinitesimal solitonic deformations. Other general results on  solitonic rigidity are given, when the curvature is   sufficiently positive or the metric $g_o$ is K\"ahler.
\par
\smallskip
We conclude observing that these  results appear to be the  natural analogues of the outcomes of the  theory of deformations of Einstein metrics,  developed by Koiso (\cite{Ko, Ko1, Ko2, Be}), albeit
they are based on  various  non trivial  modifications of Koiso's settings and techniques.
\par
\smallskip
The structure of the paper  is as follows. In \S 2,  we consider the set of  normalised solitons inside  $\cP^s$,  define and prove the existence of  $f$-twisted slices and  give the notion of solitonic pre-moduli space. In \S 3,  we define the space of essential infinitesimal solitonic deformations, we  show that it is finite dimensional and we prove   our main result, which gives  a  local description of the solitonic pre-moduli space. In \S 4,  we introduce the notion of  the solitonic rigidity and  study the rigidity of the standard metrics of a compact symmetric spaces of rank one,   showing  that the space of essential infinitesimal solitonic deformations of $\bC P^n$ is non trivial. We also treat the solitonic rigidity of Einstein metrics with certain  conditions on the curvature. In  \S 5,  we define the weak solitonic rigidity of Einstein metrics and  obtain  some related results under  conditions on the diameter of the manifold or in the K\"ahler situation. \par
\smallskip
\noindent{\it Notation.} \par
Throughout the paper,  $M$ is  a compact $n$-dimensional manifold,  $\cF(M) = \cC^\infty(M, \bR)$ is  the space of $\cC^\infty$ real  functions, $\cD$ is the set of all diffeomorphisms from $M$ into itself and $\cM$ is the space of  all Riemannian metrics on $M$, i.e.,  the cone of  smooth sections of  $S^2(T^*M)$ that  determine  positive definite inner products on each tangent space of $M$. \par

 For a fixed Riemannian metric $g$ on $M$, we  denote by $\DD$ its Levi Civita connection and by $\R$ its Riemann curvature tensor, defined by
 $\R_{XY} = \DD_{[X,Y]} - [\DD_X,\DD_Y]$ for any pair of vector fields $X$,$Y$. The Ricci tensor,  the scalar curvature and the volume form of $g$ are   denoted by $\Ric$, $s$ and $\m$,  respectively.
 When it is  necessary to specify the dependence on the given metric $g$, a superscript or subscript $``g"$  is sometimes added.\par

We denote by  $\Hs (M)$ the Sobolev space of  real functions  with  partial derivatives in $L^2(M)$ up to order $s$.
  Similarly,  if $\pi: E \longrightarrow M$ is a vector bundle,  $\Hs(E)$ is the  space of  sections of $E$ with square integrable partial derivatives up to order $s$, so that   $\cC^\infty(E) = \bigcup_{s \geq 0} \Hs(E)$ is the space of smooth sections of $E$.\par
 Recall that, by   Sobolev embedding's Theorem,   $\Hs$-differentiability implies $\cC^{s-\left[\frac{n}{2}\right]-1}$-differentiability.   This allows to consider
for any $s > \frac{n}{2}$  the collection  $ \cM^s$ of $\cC^0$ Riemannian metrics
that are in $\Hs(S^2(T^*M))$,  and  the group  $\cD^{s+1}$ of $\operatorname{H}^{s+1}$-diffeomorphisms of $M$, i.e. of  $\cC^1$-diffeomorphisms $\varphi$, with coordinate expressions   of $\varphi$ and $\varphi^{-1}$  both in $\sob {s+1}$.
As it is shown in   \cite{Eb, Eb1},    $\cM^s$  and $\cD^{s+1}$ are   Hilbert manifolds that naturally include $\cM$ and  $\cD$, respectively,  and  $\cD^{s+1}$ is a topological group acting  on $\cM^s$  via a right action,  which   extends  the standard  right action of $\cD$ on $\cM$,  namely $\varphi \ast g\= \varphi^*(g)$ with $\varphi\in \cD^{s+1}$ and $g\in \cM^s$.

We finally recall that  the tangent spaces $T_g\mathcal M^s$  and $T_\varphi \cD^{s+1}$ of $\cM^s$ and $\cD^{s+1}$, respectively,  are naturally identifiable with the Hilbert spaces
 \beq \label{identifications} T_g\cM^s \simeq \Hs(S^2(T^*M))\ ,\qquad T_\varphi \cD^{s+1}Ê\simeq \sob{s+1}(TM)\eeq
 and that the scalar products
 \beq \label{1.2} <\cdot, \cdot>_g: T_g \cM^s \times T_g\cM^s \longrightarrow \bR\ ,\ \ <h, k>_g \= \int_M g(h, k) \m_g\eeq
 determine a $\cD^{s+1}$-invariant, smooth (weak) Riemannian structure on $\cM^s$.
 \par
\bigskip
\section{Moduli and pre-moduli spaces of Ricci solitons}
\setcounter{equation}{0}
\subsection{Normalised Ricci solitons and Perelman's entropy functional}\hfill\par
Let $M$ be an  $n$-dimensional compact manifold. We recall that a  Riemannian metric $g$ on $M$ is called  {\it (gradient) Ricci soliton} if there exist $f\in {\mathcal F}(M)$ and a constant $c\in \bR$ so that
\beq \label{definitionsol}\Ric + \DD d f - c\cdot g = 0\ .\eeq
If $g$ is a Ricci soliton,  there exists a unique $c$ and a unique $f$, determined up to a constant,  satisfying \eqref{definitionsol}. The Einstein metrics are  Ricci solitons of a special kind,  with  $f$ constant. The  non-Einstein Ricci solitons are  called  {\it non-trivial}.   \par
It is   known  that the  non-trivial Ricci solitons on  compact manifolds are  necessarily  {\it shrinking}, i.e. with $c>0$ (see e.g. \cite{ELM}), so  that any  Ricci soliton on $M$, which is not Ricci flat,  can be rescaled   to have    \eqref{definitionsol}Ê satisfied with the constant $c=1$.\par
\smallskip
In particular, if we consider the   set of pairs
 $$\cP \= \left\{ (g,f)\in \cM \times \cF(M)\ : \ \  \frac 1{(2\pi)^{\frac n2}}\, \int_M e^{-f}\, \mu_g = 1\  \right\}\ ,  $$
 we have that  the equivalence classes  up to
homotheties of the Ricci solitons that  are not Ricci flat,  are in one-to-one correspondence with the  subset of $\cP$ defined by
$$\Sol \= \left\{\ (g,f)\in \cP\ :\  \Ric_g + \DD^gdf - g = 0\ \right\}\ .$$
We call the elements of this space  {\it normalised Ricci solitons\/}.\par
\medskip
Consider  the map $W: \cP \longrightarrow \bR$,   given by the restriction to  triples $(g, f, \frac{1}{2}) \in \cP \times \{(\frac{1}{2})\}$ of  Perelman's entropy functional  (\cite{Pe}), i.e.
\beq \label{entropy}W(g, f) =  \frac{1}{(2 \pi)^{\frac{n}{2}}} \int_M \left(\frac{1}{2} |\nabla f|_g^2 +\frac{1}{2}s_g  + f - n\right) e^{-f} \mu_g\ .\eeq
 The Euler-Lagrange equations for   its   critical points  are
 (see e.g. \cite{HM, Ch})
\beq \label{ELequations} \left\{\begin{array}{l} \Ric_g  +  \DD^gdf - g = 0\ ,\\
\ \\
\Delta_g f - \frac{1}{2}Ê|\nabla f|_g^2 + \frac{1}{2}Ê s_g + f - n  -W(g,f) = 0\ ,
\end{array}\right.\eeq
so that any such critical point is  in $\Sol$ (it satisfies  the first of \eqref{ELequations}).
Conversely, if $(g, f) \in \Sol$, a  well known argument (see e.g. \cite{ELM}) shows that the function $\Delta_g f - \frac{1}{2}Ê|\nabla f|_g^2 + \frac{1}{2}Ê s_g + f$ is equal to a constant, which,  by integration, is directly  seen to be   $n + W(g, f)$. This means  that   the following  conditions are equivalent:
{\it \begin{itemize}
\item[a)] $(g,f) \in \Sol$;
\item[b)] $(g,f)$ is a solution  of  \eqref{ELequations};
\item[c)] $(g,f)$ is a critical point of Perelman's  entropy functional \eqref{entropy},
\end{itemize}}
\noindent
and indicates  a useful  parallelism  between Ricci solitons and Einstein metrics.
In the following, we denote by $\bS $ the Euler-Lagrange operator
\beq \bS =  (\bS_1, \bS_2):\cP \longrightarrow \cC^\infty(S^2T^*M) \times \cF(M)\ ,$$
$$\label{map} \bS(g,f) = \left(\Ric_g +  \DD^gdf  - g\ ,\  \Delta_g f - \frac{1}{2} |\nabla f|_g^2 +  \frac{1}{2}  s_g + f - n - W(g,f)\right)\ ,\eeq
so that ${\Sol} = \bS^{-1}(0,0)$.
\medskip
\subsection{Slices and pre-moduli spaces of Ricci solitons}\hfill\par
Let $g \in \cM \subset \cM^s$ be a $\cC^\infty$ Riemannian metric  and   $\Iso(g) \subset \cD$ its isometry group.  We recall that  Ebin's Slice Theorem (\cite{Eb, Eb1}) states that,  for any  $s \geq \left[\frac{n}{2}\right] + 3$, there exist
\begin{itemize}
\item[--]  a $\cD^{s+1}$-invariant neighbourhood $\cU^s \subset \cM^s$ of $g$,
\item[--] a Hilbert submanifold $\cS^s \subset   \cU^s$  through $g$ (called the {\it Ebin slice}),
\item[--] a  neighbourhood $\cV^{s+1} \subset \cD^{s+1}/\Iso(g)$ of  the coset   $o = \Iso(g)$ and  a local cross section $\chi: \cV^{s+1}\longrightarrow  \cD^{s+1}$,
\end{itemize}
such that  {\it
\begin{itemize}
\item[1)] $\cS^s$ is $\Iso(g)$-invariant and   for   $\varphi  \in \cD^{s+1}$
$$\cS^s \cap ( \varphi\ast g) \neq \emptyset\qquad \Longleftrightarrow \qquad \varphi\in \Iso(g)\ ;$$
\item[2)] the map
$$L: \cV^{s+1} \times \cS^s  \longrightarrow  \cM^s\ ,\qquad L(u, h) = \chi(u)\ast h$$
determines  a homeomorphism between $\cS^s \times \cV^{s+1}$ and  $\cU^s$.
\end{itemize}}
This  implies that    {\it the  space $ \cM^s/\cD^{s+1}$ of isometry classes  in $\cM^s$ has an open neighbourhood  $\overline \cU^s = \cU^s/\cD^{s+1} $ of  $ [g] = \cD^{s+1}\ast g$, which  is  naturally   identifiable with the quotient  $\cS^s/\Iso(g)$}.
\par
\bigskip
We remark that the Ebin slice $\cS^s$ is in fact  the image under the exponential map $\exp: T_g \cM^s \longrightarrow \cM^s$
of an open subset of  a  subspace $V^s\subset T_g \cM^s$, which is complementary to
 the tangent space  $T_g( \cD^{s+1}\ast g)$. Such complementary   subspace  can be shortly described as  follows.
Consider the operators
\beq
\begin{array}{ll}\!\!\!\!\!\!\a_g: \sob{s+1}(TM)\longrightarrow \Hs(S^2T^*M), & \a_g(X) \:= \cL_X g\ ,\\
\ & \\
\!\!\!\!\!\!\d_g: \Hs(S^2T^*M)\to  \sob{s-1}(T^*M)\!\simeq\! \sob{s-1}(T M),& \d_g h_x := (\DD_{e_i}h)_x(e_i,\cdot)\end{array}\eeq
 where  $\cL_X$ is the Lie derivative along the vector field $X$ and
 $(e_i)$ is an arbitrary  $g_x$-orthonormal basis of $T_xM$, $x\in M$.  One can  check that $\a_g$ has injective symbol and that $- \d_g$ is  the  adjoint  $\a^*_g$ of $\a_g$  w.r.t.
 \eqref{1.2}, i.e.
 $$<\a_g(X), k>_g = \int_M  g(\cL_X g, k) \m_g = - \int_M  \d_gk(X) \m_g = <X, \a^*_g(k) >_g\ .$$
From this, one gets  the  $<,>_g$-orthogonal decomposition  (\cite{Eb1}, Cor. 6.9)
\beq \label{dec}  \Hs(S^2T^*M) = \Im \a_g|_{\sob{s+1}(TM)} \oplus \ker \d_g|_{\Hs(S^2T^*M)}\ . \eeq
Under  the  natural identification $T_g \cM^s =  \Hs(S^2T^*M)$,   the subspace
 $T_g( \cD^{s+1}\ast g) \subset T_g \cM^s$   corresponds to  $\Im \a_g|_{\sob{s+1}(TM)}$ and the subspace $V^s$ can be chosen to be
 $\ker \d_g|_{\Hs(S^2T^*M)}$, which is then the tangent space $T_g \cS^s$ at $g$.
\par
\medskip
Imitating Ebin's construction,  for any  Ricci soliton $(g, f) \in \cP$ we now want to construct a special  submanifold $\cS_{f}^s \subset \cM^s$ with  the same property of the Ebin slice,  but with a tangent space $T_g \cS_{f}^s \subset T_g \cM^s$  characterised in a different way, more convenient  for our purposes.  \par
\smallskip
Let $g \in \cM^s$ and  assume that  $f \in \cF(M)$ is a  smooth function, which is $\Iso(g)$-invariant  (this surely occurs when $(g,f)$ is a Ricci soliton). The corresponding {\it  twisted divergence}Ê is the operator
\beq \d_{(g,f)}:\Hs(S^2T^*M)\to \sob{s-1}(T^*M)\ ,\ \d_{(g,f)} h \= e^f\cdot \d_g(e^{-f}\cdot h) = \d_g h - \imath_{\nabla f}h\ .\eeq
When $g$ is considered  as known,  we will shortly write  $\d_f$ instead of  $\d_{(g, f)}$. \par
\smallskip
Notice that, in analogy with $-\d$, the operator $- \d_f$ is  the formal adjoint  of $\alpha_g$ w.r.t. the ``twisted'' inner product on $T_g \cM^s = \Hs(S^2(T^* M))$
$$<h,k>_{(g,f)} := \int_M g(h,k) e^{-f}\mu_g\ .$$
We have that
\begin{prop} For  $s\geq \left[\frac{n}{2}\right] + 3$ % the following $<,>_{(g,f)}$-orthogonal decomposition holds
\beq\label{twisteddec}  \Hs(S^2T^*M) = \a_g(H^{s+1}(TM)) \oplus \ker \d_{(g,f)}|_{\Hs(S^2T^*M)} \eeq
and there exists a submanifold $\mathcal S^s_{f} \subset \cM^s$, which satisfies   the same conditions  (1) and (2) of  $\cS^s$, but
 has tangent space $T_g\cS^s_{f} = \ker \d_{(g,f)}|_{\Hs(S^2T^*M)}$.
\end{prop}
\begin{pf} We    follow the line of arguments  of \cite{Eb1}, Thm.7.1. Since $\a$ has injective symbol and $-\d_{(g,f)}$ is the formal adjoint of $\a_g$ w.r.t. the inner product $<,>_{(g,f)}$, one can check  that  $\d_{(g,f)}\circ\a_g$ is  elliptic and infer that  \eqref{twisteddec} holds (see e.g.  \cite{Eb1}, Cor. 6.9).\par
Consider now the orbit $\cO^s\=  \mathcal D^{s+1}\ast g$ in $\cM^s$ and the restricted tangent bundle $T(\mathcal M^s)|_{\cO^s}$ over $\cO^s$. For any $ g' =\varphi^*(g) \in \cO^s$, consider the subspace of $T_{g'} \cM^s$ defined by
$$\nu_{g'} = \varphi^*(\ker \d_{(g,f)}) \subset \Hs(S^2 T^* M) = T_{g'} \cM^s$$
and set  $\displaystyle \nu = \bigcup_{g' \in \cO^s}Ê\nu_{g'}$.  Note that the  $\nu_{g'}$'s (and hence $\nu$) are well defined,  because $f$ is $\Iso(g)$-invariant. \par
We claim that $\pi: \nu \to \cO^s$  is a smooth  subbundle of $\pi: T\cM^s|_{\cO^s} \to \cO^s$. In fact, one can consider the smooth family of inner products $< , >_{(g', f')}$, determined by  $ g' = \varphi^*(g) \in \cO^s$ and  $f' = \varphi^*(f)$,   and the corresponding family of orthogonal projectors
$$P_{g'}: T_{g'} \cM^s \longrightarrow T_{g'} \cO^s = \Im \a_{g'}\ ,\quad P_{g'} \= \a\circ (\d_{f'}\circ \a)^{-1}\circ \d_{f'} ,$$
where, for shortness,  we denote by  $\a \= \a_g$ and $\d_{f'} \= \d_{(g', f')}$.  The map $P_{g'}$ is actually well defined because the composition
$$(\d_{f'}\circ \a)^{-1}\circ \d_{f'}|_{\Hs(S^2T^*M)} : \Hs(S^2 T^* M) \longrightarrow \sob{s+1} (TM)$$
is a single valued map: to see this, one needs  to observe that
$$\d_{f'}(\Hs(S^2T^*M)) = (\d_{f'}\circ \a)(\sob{s+1}(TM))$$
 and that  the self-adjoint elliptic operator $\d_{f'}\circ \a: \sob{s+1}(TM)\longrightarrow \sob{s-1}(TM)$ induces a bijection
$$\left.\d_{f'}\circ \a\right|_{(\d_{f'}\circ \a)(\sob{s+1}(TM))}: (\d_{f'}\circ \a)(\sob{s+1}(TM))\overset{\sim}\longrightarrow (\d_{f'}\circ \a)(\sob{s-1}(TM))\ .$$
Consider now   the  projection $P:  T\cM^s|_{\cO^s} \to T\cO^s$, given at any $g' \in \cO^s$ by the operator $P_{g'}$ and observe that $\nu = \ker P$. Following   exactly the same arguments in \cite{Eb1}, p. 31--32, one can  check that $P$ is a smooth map and hence   that $\pi: \nu = \ker P \to \cO^s$ is a smooth bundle, as claimed.
\par
\smallskip
Now, if  $\exp: T\cM^s \longrightarrow \cM^s$ is the exponential map relative to the  weak Riemannian metric \eqref{1.2}, there exists
a neighbourhood $\cV \subset \nu$ of the zero section, such that  $\exp|_{\cV}: \cV \longrightarrow \cU \subset \cM^s$ is a diffeomorphism onto  a $\cD^{s+1}$-invariant neighbourhood $\cU$ of  $\cO^s$. By the same arguments of \cite{Eb}, p. 32--33, the submanifold $\cS^s_{f} = \exp(\cV \cap \nu|_g) $ satisfies all required conditions. \end{pf}
\medskip
We call  $\cS^s_{f}$ the  {\it  $f$-twisted slice at $g$ in $\cM^s$};  when $g$ is smooth, we  set $\mathcal S^\infty_{f} \= \mathcal S^s_{f} \cap \mathcal M$.  As in \cite{Eb} Thm. 7.4, one  can  check  that also the set $\mathcal S^\infty_{f}$  satisfies the  conditions (1) and (2) of  Ebin slices.\par
\smallskip
For any $s\geq \left[\frac{n}{2}\right]+3$, we denote by $\cP^s$ the Hilbert manifold
$$\cP^s :=  \left\{ (g,f)\in \cM^s \times \Hs(M)\ |\ \frac1{(2\pi)^{\frac n2}}\int_M e^{-f} \mu_g = 1\ \right\}\ , $$
which naturally extends $\cP$ and we obviously  extend  \eqref{map} to the operator
$$\bS^s: \cP^s \longrightarrow \sob{s-2}(S^2T^*M) \times \sob{s-2}(M)\ .$$
We also set  $\Sol^s \= (\bS^s)^{-1}(0)$.  Since  any  Ricci soliton  of class $\cC^2$ is necessarily  real analytic (\cite{DW}), by standard arguments one gets  that
$\Sol^s = \Sol$ for every $s\geq \left[\frac{n}{2}\right]+3$, in complete analogy with  the Einstein case  (see e.g. \cite{Ko1}, Lemma 2.5). \par
\medskip
Let  $(g,f)\in {\Sol}$ and recall  that   $f$ is $\Iso(g)$-invariant, so that  one can consider the $f$-twisted slice $\cS^s_{f} \subset \cM^s$ through $g$ and the subset
$\mathcal S^\infty_{f} = \cS^s_{f} \cap \cM$ of $\cM$. In analogy with \cite{Ko1} (see also \cite{Be}, Ch. 12), we introduce the following:
\begin{definition} A {\it  solitonic pre-moduli  space  at $(g,f)$\/} is the set
$$\cA_{(g,f)} \= {\Sol}^s \cap \left(\cS^s_{f} \times \Hs(M)\right) =  {\Sol} \cap \left(\cS^\infty_{f}\times\cF(M)\right) $$
where $\cS^s_{f}$ is a $f$-twisted slice through $g$ and $s\geq \left[\frac{n}{2}\right]+3$.
\end{definition}
Note that  $\cA_{(g,f)}$ is invariant under the natural  action of $\Iso(g)$ on  $\cP^s$. Moreover,
from  the  properties of slices,   for any $s \geq  \left[\frac{n}{2}\right]+3$, the quotient $(\Sol \cap \cU^s)/\cD^{s+1}$ of a   sufficiently small   neighbourhood $\Sol \cap \cU^s$  of   $(g, f)\in \Sol^s$ can be naturally identified with the  quotient $\cA_{(g,f)}/\Iso(g)$ of the corresponding solitonic pre-moduli space $\cA_{(g,f)}$.  This means that  {\it the  local behavior of the moduli space $\Sol/\cD$ is determined by  the  quotients of the pre-moduli spaces  by   the actions of (finite dimensional!) groups  of isometries}. \par
\bigskip
\section{The solitonic pre-moduli spaces are\\
 real analytic sets in finite-dimensional manifolds}
\setcounter{equation}{0}
Consider  a $\cC^1$-curve  $(g_t,f_t)$ in $\Sol$  passing through a fixed Ricci soliton $(g_0, f_0) = (g, f)$. The tangent vector $(h = \dot g|_{t = 0}, a = \dot f|_{t = 0})$ is necessarily  in the kernel of  the linearized operator $d\bS_{(g,f)}$ of the operator \eqref{map}.  If furthermore   $(g_t,f_t)$ takes values in a pre-moduli space $\cA_{(g,f)}$,  its tangent  vector $(h, a)$  satisfies the additional condition $\d_{(g, f)} h = 0$.  These observations lead to the following definition.
\begin{definition} For a given $(g,f)\in {\Sol}$,  the elements  of the subspace $ \ker d\bS_{(g,f)}\subset  T_{(g,f)}\cP$ are called  {\it infinitesimal solitonic deformations} (shortly,  {\it i.s.d.}) {\it of $(g,f)$}.  The i.s.d.'s  $(h,a)$ such that   $\d_{(g,f)} h = 0$  are called  {\it essential}.
\end{definition}
Let us  denote by  $Z_{(g,f)}$ the  space of the essential i.s.d.'s at $(g,f)$  and  set
$$Z^s_{(g,f)} = \ker d\bS^s_{(g,f)}\cap (T_{(g,f)}\cS^s_{f} \times \Hs(M))\ .$$
 We will shortly  see that  $ Z^s_{(g,f)}$ is actually identifiable with   $Z_{(g,f)}$.
\par
\smallskip
In the next lemma, we give the explicit expression for the linearization   $d\bS_{(g,f)}$ of the operator $\bS = (\bS_1,\bS_2)$ at a normalised Ricci soliton $(g,f)$. In the following formulas,
 $\Delta_f$  is the   twisted Laplacian (also called Bakry-Emery Laplacian or Witten Laplacian),  acting on symmetric $2$-tensors as
 $$\Delta_f h = \Tr(\DD^2h) - \DD_{\nabla f}h$$
  and   $\cR: \cC^\infty(S^2T^*M)\longrightarrow \cC^\infty(S^2T^*M)$ is   the operator
 $$\cR(h)(X,Y) = \Tr(h(\R_{X \cdot}Y,\cdot))\ .$$
\begin{lemma}\label{lemma} For any  $(h,a)\in T_{(g,f)}\cP$
\beq\label{dS1} d\bS_1|_{(g,f)}(h,a) = -\frac 12 \Delta_f h - \mathcal R(h) - \frac 12 Dd(\Tr(h)-2a) + \frac 12 \cL_{(\d_f h)^\sharp} g\ ,                   \eeq
\beq\label{dS2} d\bS_2|_{(g,f)}(h,a) = \Delta_f(\Tr(h) - 2a) + (\Tr(h)-2a) - \d_f\left( \d_f(h)\right)\ .   \eeq
\end{lemma}
\begin{pf}  By the classical formulas for the  variations of  Ricci tensors and Hessians (see e.g. \cite{Be,To}) we have that
\beq\label{eq}2\, d\bS_1|_{(g,f)}(h,a) = -\Delta h - \DD d\Tr(h) +  \cL_{(\d h)^\sharp}g - 2\cR(h) + \phantom{aaaaaaaaaaaaaaa} $$
$$\phantom{aaaaaaaaaaa} +\Ric\circ h + h\circ \Ric - 2 h + 2\DD da + \DD_{\nabla f}h\  - [\DD h\cdot \nabla f]\ ,\eeq
where $[\DD h\cdot \nabla f]$ is the symmetric tensor
$$[\DD h\cdot \nabla f](X,Y) = \DD_Xh(\nabla f,Y) + \DD_Yh(\nabla f,X)$$
and $\Ric\circ h$ denotes  the $(2,0)$-tensor,  associated by the metric $g$ to the composition of $\Ric$ and $h$,
viewed as $(1,1)$-tensors.  From Ricci soliton equation, the definition of  $\d_f h$ and the fact that
$\cL_{(\imath_{\nabla f}h)^\sharp}g  = [\DD h\cdot \nabla f] $, equation   \eqref{dS1} follows. Equation \eqref{dS2} follows from the computations in \cite{HM}, p. 3332, together with the fact that $dW|_{(g,f)} = 0$ because $(g,f) \in \cP$ is a   soliton.\end{pf}
From Lemma \ref{lemma}, we directly get the   following
\begin{theorem}\label{thm1} For any    $(g,f) \in \Sol$,  the  space  $Z_{(g,f)}$ is given by
\beq\label{Z} Z_{(g,f)}  =  \left\{ (h,a)\in T_{(g,f)}\cP\, |\, \d_f h = 0,\, \frac 12 \Delta_f h + \cR(h) = 0,\, a= \frac {\Tr(h)}2\right\}\eeq
and it is finite dimensional.
\end{theorem}
\begin{pf} If $(h,a)\in Z_{(g,f)}$, the function $u:= \Tr(h) - 2a$ satisfies $\D_f u + u = 0$ by \eqref{dS2}. Since the equation  $\D_f u + \l u = 0$ admits  non-trivial solutions  only if  $\l> 1$ (\cite{CZ}),  $u = 0$ and,  by
definitions and Lemma \ref{lemma},   \eqref{Z} follows. Being  the operator $\frac 12 \D_f + \cR$ on $\cC^\infty(S^2T^*M)$  elliptic,
 $Z_{(g,f)}$ is finite dimensional. \end{pf}
We remark here that $Z^s_{(g,f)} = Z_{(g,f)}$. Indeed the same argument as in the proof of Theorem \ref{thm1} shows that $Z^s_{(g,f)}$ consists of elements
$(h,a)$ in $T_{(g,f)}\cP^s$ such that $2a= \Tr(h)$ and $\frac 12 \Delta_f h + \cR(h) = 0$, hence $h$ and $a$ are smooth.\par
\smallskip
The next theorem is an analogue of \cite{Ko1}, Thm. 3.1 on Einstein metrics  and is  a crucial  property  of  solitonic pre-moduli spaces.
\begin{theorem}\label{thm2} Let $(g,f)\in {\Sol}$ be a normalised Ricci soliton and denote by $\cS^s_{f} $  an $f$-twisted slice   at  $g$ in $\cM^s$, $s\geq \left[\frac{n}{2}\right]+3$.
Then   there exists an open neighbourhood $\cU^s$ of $(g,f)$ in $(\cS^s_{f} \times \Hs(M))\cap \cP^s$ and a finite dimensional, real analytic submanifold $\cZ^s\subset \cU^s$ such that :
\begin{itemize}
\item[i)]   $T_{(g,f)}\cZ^s = Z_{(g,f)}$;
\item[ii)] the  pre-moduli space $\cA_{(g,f)} = (\cS_{f}^s\cap \Hs(M)) \cap \Sol$ is a real analytic subset of $\cZ^s$.
\end{itemize}

\end{theorem}
\begin{pf}  From  definitions,
$ \cA_{(g,f)} = (\bS^s|_{(\cS^s_{f} \times \Hs(M))\cap \cP^s})^{-1}(0,0)$ and    if
we can show that the image
$$d\bS^s|_{(g,f)}(T_{(g,f)}((\cS^s_{f} \times \Hs(M))\cap \cP^s))$$
 is a closed subspace of $\sob{s-2}(S^2T^*M)\times \sob{s-2}(M)$, the claim follows from the Implicit Function Theorem in Hilbert spaces and  the same arguments of the proof of  \cite{Ko1}, Thm. 3.1. Now, we observe   that the tangent space
$$V^s_{(g,f)} := T_{(g,f)}((\cS^s_{f} \times \Hs(M))\cap \cP^s) = $$
$$= \left\{\ (h,a)\in \Hs(S^2T^*M)\times \Hs(M)\,|\, \d_fh = 0,\, \int_M(\Tr(h)-2a)\, e^{-f}\mu_g = 0\ \right\}$$
is identifiable with the vector space
$$W^s_{(g,f)} \= \left\{\ (h,u)\in \Hs(S^2T^*M)\times \Hs(M)\,|\, \d_fh = 0,\, \int_M\, u\, e^{-f}\mu_g = 0\ \right\} \ . $$
%where the last isomorphism is given by $(h,a)\mapsto (h,u:= \Tr(h)-2a)$.
We also note that
$$d\bS^s(V^s_{(g,f)}) = F^s(W^s_{(g,f)})\ ,$$
 where $F$ denotes the elliptic differential operator
$$ F^s: \Hs(S^2T^*M)\times \Hs(M) \longrightarrow \sob{s-2}(S^2T^*M)\times \sob{s-2}(M)$$
$$F^s(h,u) = \left(-\frac 12 \Delta_f h - \mathcal R(h) - \frac 12 \DD du\ ,\ \Delta_fu + u\right)\ .$$
Hence, the proof  reduces to showing that $ F^s(W_{(g,f)})$ is a closed subspace of $\sob{s-2}(S^2T^*M)\times \sob{s-2}(M)$. \par
\medskip
Consider  the linear differential operator
$$\b_g: \Hs(S^2T^*M)\times \Hs(M) \longrightarrow \sob{s-1}(T^*M)\ ,\qquad \beta_g(h,u) \= \d_fh - \frac 12 du\, .$$
\begin{lemma} For every $g\in \mathcal M^s$ and $f\in \Hs(M)$, with $s\geq \left[\frac{n}{2}\right]+3$,
\beq \d_f(\Ric + \DD df - g) = \frac 12\, d(2\,\Delta f - |\nabla f|^2 + s + 2f).\eeq
\end{lemma}
\begin{pfns} The claim is a  consequence of   the contracted second Bianchi identity $2\ \d \Ric = ds$ and  the  equalities
$$d(|\nabla f|^2) = 2\, \d(\DD df)\ ,\qquad d\Delta f = \d(\DD df) - \imath_{\nabla f}\Ric.\eqno\square$$\end{pfns}
As a corollary we have that for every $(g',u)\in\cP^s$
\beq \label{comp} \b_g \circ \bS^s (g',u) = 0.\eeq
So,  taking the differential  of \eqref{comp} at  $(g,f)$ and by the fact that  $\bS(g,f)=0$,
\beq\label{comp2} \b_g\circ d\bS^s|_{(g,f)}(h,a)= 0\qquad\qquad \text{for any} \ (h,a) \in T_{(g,f)} \cP^s\ .\eeq
Using Lemma \ref{lemma},  the equality  \eqref{comp2} can be rewritten as
\beq\label{comm}  \b_g\circ F^s(h,\Tr(h)-2a) = -\d_f(\frac 12 \cL_{(\d_f h)^\sharp)}g)  - \frac 12 d\d_f\d_f(h) = G^{s-1}(\d_fh),\eeq
where we indicate by $G^s$ the elliptic operator
$$G^s:\Hs(T^*M)\longrightarrow \sob{s-2}(T^*M)\ ,\quad G^s(\omega) = -\frac 12(\d_f(\cL_{\o^\sharp}g) + d\d_f\o)\ .$$
Now by \eqref{comm} we have
\beq\label{incl}  F^s(W_{(g,f)}) \subseteq \ker\b_g \cap \operatorname{Im} F^s.\eeq
Notice also that $\ker\b_g \cap \operatorname{Im} F^s$ is a closed subspace,  because $ F^s$ is elliptic.\par
On the other hand, given $(k,v)\in \ker\b_g \cap \operatorname{Im} F^s$, we may write $(k,v) = F^s(h,u)$ for some $(h,u)\in \Hs(S^2T^*M)\times \Hs(M)$.
Using \eqref{dec} we can find some $X\in H^{s+1}(TM)$ such that $h= \a(X) + h_1$ with $\a(X):= \cL_X g$ and $\d_f h_1= 0$.
Moreover we can decompose $u = u_o + u_1$, with
$u_o\in \bR$ and $u_1$ satisfying $\int_M u_1e^{-f}\mu_g = 0$.
Since $(k,v)\in \ker\b_g$, we have that $G^s(\d_fk) = 0$ and hence $G^s(\d_f(\a(X))) = 0$, i.e. $X \in V\= \ker
(G^s\circ\d_f\circ\a) \subset H^{s+1}(TM)$, which  is a finite dimensional space being $G^s$  elliptic. Substituting,  we get that
$$(k,v) =  F^s(h,u) =  F^s(h_1,u_1) +  F^s(\a(X),u_o)\ ,$$
which means that
\beq \label{incl2} \ker\b_g \cap \operatorname{Im} F^s \subseteq    F^s(W_{(g,f)}) +  F^s(\a(V)\times \bR)\ ,\eeq
where $F^s(\a(V)\times \bR)$ is a finite dimensional space. From \eqref{incl} and \eqref{incl2}, it follows that
$ F^s(W_{(g,f)})$ has finite codimension in the closed subspace $\ker\b_g \cap \operatorname{Im} F^s$. Since $F^s(W_{(g,f)})$ is
the image of a bounded linear operator, a standard argument shows that $ F^s(W_{(g,f)})$ is closed (see e.g. \cite{Pa1}, p. 119).
\end{pf}
\bigskip
\section{Solitonic rigidity}
\setcounter{equation}{0}
In  the following, we  constantly assume  $s \geq  \left[\frac{n}{2}\right]+3$.
\par
\begin{definition} A normalised Ricci soliton $(g,f)$ is said to be  {\it solitonic rigid\/} or, shortly,  {\it sol-rigid\/}  (resp. {\it solitonic rigid in $\cP^s$}) if there exists a neighbourhood $\cU\subset \cP$ (resp. $\cU \subset \cP^s$) of $(g,f)$  such that  $\cU \cap \Sol$ consists   only of the $\cD$-orbit   (resp.  $\cD^{s+1}$-orbit) of $(g,f)$.
\end{definition}
By a classical result of Palais (\cite{Pa}; see also \cite{Ko} p. 53) if two normalised Ricci solitons are in the same $\cD^{s+1}$-orbit,  they both lie in the same $\cD$-orbit. So,   if  $(g, f)$ is solitonic rigid in $\cP^s$,   it is  automatically   sol-rigid. \par
\begin{example}\label{Ex}
The constant curvature metrics  of the standard sphere $S^n$ and of the real projective space $\bR P^n$ are  sol-rigid. In fact,
by a result of B\"ohm and Wilking (\cite{BW}), the shrinking Ricci solitons with $2$-positive curvature have constant sectional curvature, so that
any  Ricci soliton on $S^n$ or $\bR P^n$,  which is close  to the standard metric $g_o$,  is surely  isometric to $g_o$.
\end{example}
\medskip
The next proposition is  a direct consequence of  the notion of pre-moduli spaces  and Theorem \ref{thm2} and gives a useful tool for proving  sol-rigidity.\par
\begin{proposition} Let $(g, f)$ be a normalised Ricci soliton. If the space $Z_{(g,f)}$  of essential i.s.d.'s is trivial, the soliton  $(g,f)$ is   sol-rigid.\end{proposition}
Recall  now that an Einstein metric $g$ on $M$ with Einstein constant $c = 1$ corresponds to   a normalised Ricci soliton  $(g, f ) \in \Sol$ with
$f$ constant and equal to  $f = \log\left(\frac{\vol(M,g)}{(2\pi)^{n/2}}\right)$. So, the set $\cE$ of all (normalised) Einstein metrics can be identified with the subset of  $\cP$  given by
\beq\label{einstein} \cE := \left\{\ \left(g, -\log\left(\frac{\vol(M,g)}{(2\pi)^{n/2}}\right)\right) \in \cP\ |  \Ric_{g} - g = 0\ \right\} =\eeq
 $$= \Sol \cap \{(g,u)\in \cP|\, u = \text{constant}\}.$$
For any $(g, f) \in \cE$, we call  {\it space of essential infinitesimal Einstein deformations of $(g,f)$\/}  the subspace of $T_{(g,f)} \cP$
\beq E_{(g,f)} := Z_{(g,f)} \cap \left\{\ (h,a)\in T_{(g,f)}\cP\ |\ da = 0\ \right\}\ .\eeq
Note that it can be naturally identified with the spaces of essential EID  considered by Koiso (\cite{Ko1}, Def. 1.4). \par
\medskip
From Theorem \ref{thm1}, we have that if $(h,a)\in E_{(g,f)}$,  then  $\Tr(h)$ is constant and
\beq \label{traccina}\Tr(\Delta h + 2\cR(h)) = \Delta \Tr(h) + 2 \Tr(h) =  2 \Tr(h) = 0\ .\eeq
%Since  $g$ is  Einstein with Einstein constant $c = 1$, a  classical result of Lichnerowicz (\cite{Lic}) implies that the least eigenvalue of the Hodge Laplacian ($=-\Delta$) acting on functions  is not less than $\frac n{n-1} > 1$.
 This implies
 %$\Tr(h) = 0$ and hence that
\beq E_{(g,f)} = \left\{\ (h,0)\in T_{(g,f)}\cP\ |\   \Delta h + 2 \mathcal R(h) = 0 \quad \text{and}\quad  a =  \Tr(h) = 0\  \right\} \cong \eeq
$$\cong \{h\in \cC^\infty(S^2T^*M)|\, \d_g h = 0,\ \Tr(h) = 0,\ \Delta h + 2 \mathcal R(h) = 0\ \},$$
recovering the classical results on deformations of Einstein metrics (see e.g. \cite{BE} Lemma 7.1, \cite{Ko1} Lemma 1.5 or  \cite{Be},  Ch. 12).\par
\bigskip
The following proposition has useful applications
\begin{proposition}\label{Z=ZE} Let $(g,f)\in \cE$. If $2\not\in \mbox{Spec}(-\Delta,\cF(M))$, then
$$Z_{(g,f)} = E_{(g,f)}.$$
\end{proposition}
\begin{pf} From \eqref{traccina} and the hypotheses,  for any $(h, a) \in  Z_{(g,f)}$, we have $\Tr(h) = 0$. Since  $E_{(g,f)} = \{(h,a)\in Z_{(g,f)}|\, \Tr(h) = 0\}$, the claim follows.\end{pf}
Next theorem is a remarkable  example of how Proposition \ref{Z=ZE} can be used.\par
\begin{theorem}\label{cross} Let $M = G/H$ be a compact rank one symmetric space and $g_o$  its standard Einstein metric, corresponding to the normalised Ricci soliton $\left(g_o, f_o = \log\left(\frac{\vol(M,g_o)}{(2\pi)^{\frac{n}{2}}}\right)\right)$. \par
\begin{itemize}
\item[a)] If $M \neq \bC P^n$, $n \geq 2$, then  $ (g_o, f_o)$  is sol-rigid.
\item[b)] If $M=\bC P^n = \SU_{n+1}/\mathrm{S}(\U_1\times\U_n)$,  $n \geq 2$, the space of essential infinitesimal Einstein deformations $E_{(g_o, f_o)}$ is trivial, while the space of
essential infinitesimal solitonic deformations  $Z_{(g_o, f_o)}$  is $\SU_{n+1}$-equivariantly isomorphic to  $\su_{n+1}$.
\end{itemize}
\end{theorem}
\begin{pf} (a) The sol-rigidity of   $M = S^n$ and $M = \bR P^n$  have been already discussed in  Example \ref{Ex}. For the  cases $M=\bH P^n$ and  $M = \operatorname{Ca}P^2$,
  by the results in  \cite{Ko},  we know that  $E_{(g_o, f_o)} = 0$.  We claim that  $2\not\in \mbox{Spec}(-\Delta,\cF(M))$ and this  immediately implies  (a) by  Proposition \ref{Z=ZE} and Theorem \ref{thm2}. This claim can be checked using  the results in \cite{CW}, where the spectrum of the Laplacian is computed for every compact rank one symmetric space.  However, in that paper, the Laplacian is computed w.r.t. the invariant metric $\overline g$ induced by the Cartan-Killing form of the Lie algebra of isometries. Since the Ricci tensor of such a metric is equal to $\frac{1}{2} g_o$,  the needed check corresponds to verify that $1 \not\in \mbox{Spec}(-\Delta_{\overline g},\cC^\infty(N))$ for $N=\bH P^n$ or $CaP^2$. From \cite{CW},  we have that
$$\mbox{Spec}(-\Delta_{\overline g},\cC^\infty(\bH P^n)) = \left\{ \frac{k^2+k(2n+1)}{2(n+2)}\ |\ k\in \bN\ \right\},$$
$$\mbox{Spec}(-\Delta_{\overline g},\cC^\infty(CaP^2)) = \left\{\ \frac{k^2+11k}{18}\ |\ k\in \bN\ \right\}$$
and the claim follows. \par
\medskip
(b) When $M=\bC P^n = \SU_{n+1}/\mathrm{S}(\U_1\times\U_n)$, from \cite{Ko}, we have that $E_{(g_o, f_o)} = 0$. It remains to determine  $Z_{(g_o,f_o)}$ and this can be done  using standard arguments of  Representation Theory and some other results in \cite{Ko}, as follows. Note  that the same arguments  determine  $Z_{(g_o,f_o)}$  for  other Hermitian symmetric spaces.\par
\smallskip
Let  $G=\SU_{n+1}$, $K= \mathrm{S}(\U_1\times\U_n)$ and  consider the Cartan decomposition $\gg = \gk \oplus\gm$ of $\gg = \su_{n+1}$,  in which $\gm$ is  naturally identified  with the tangent space  $T_{eK} M$ at the origin $ e K \in G/K $.   Denote by $\cB$ the Cartan-Killing form of $\gg$ and by $C$ the Casimir element (w.r.t. $\cB$) of the universal enveloping algebra $U(\gg)$ of $\gg$.  Recall that the $G$-invariant Riemannian metric $\overline g$   with   $\overline g|_{eK} = -B|_{\gm \times \gm}$  is the multiple of the  Fubini-Study metric $g_o$ that satisfies    $\Ric = \frac 12 \overline g$.\par
\smallskip
Since  any $G$-homogeneous vector bundle $\pi: E \longrightarrow M$, with fiber $W = E|_{eK}$,   can be naturally identified with
the bundle  $\wt \pi: G\times_K W \longrightarrow G/K$, we have
$$\cC^\infty(S^2 T^* M) \simeq \cC^\infty(G\times_KS^2(\gm^*)) \simeq \cC^\infty(G,S^2(\gm^*))_K\ ,$$
where
 $$\cC^\infty(G,S^2(\gm^*))_K := \left\{ s:G\longrightarrow S^2(\gm^*)|\ s(xy) = y^{-1}s(x),\ x\in G,\ y\in K\ \right\}\ .$$
 In particular,  $Z_{(g_o, f_o)} \simeq \ker T$, where $T$ is   the operator $T: \cC^\infty(G,S^2(\gm^*))_K \longrightarrow \cC^\infty(G,S^2(\gm^*))_K$,  corresponding to  $ \D + 2 \cR: \cC^\infty(S^2 T^* M) \longrightarrow  \cC^\infty(S^2 T^* M)$.
 \par
 \medskip
 \begin{lemma}\label{lemma1} The  finite-dimensional $G$-module $ \ker T$ is  such that
 \beq\label{U}( \ker T)^\bC\cong 2\cdot \gg^\bC.\eeq
 \end{lemma}
 \begin{pf} By  \cite{Ko}, Prop. 5.3, for any $p \geq 1$, the action of the Casimir element $C \in U(\gg)$ on the $G$-module $\cC^\infty(G,\bigotimes^p(\gm^*))_K$
 is identifiable with the differential operator
 $C =  -\Delta - 2\cR + \frac {p}{2} \Id$. This means  that $C|_{\ker T} = \Id_{\ker T}$  and, in particular,  that
   $\ker T$ has no trivial summand.
On the other hand, it is known that,   for any compact simple Lie algebra $\gk$,  the Casimir operator acts as the identity on an irreducible complex $\gk$-module $V$ if and only if  $V \simeq \gk^\bC$  (see e.g. \cite{Ko2}, p.654). Hence,  $(\ker T)^\bC \cong m\cdot \gg^\bC$ for some integer $m$. By Frobenius reciprocity,
$m = \dim_\bC \Hom_K(\gg^\bC,S^2(\gm^*)^\bC)$  and,
by standard arguments of Representation Theory,  we have $m = 2$. \end{pf}
Consider now the Hodge-Laplacian eigenspace
$F =  \{f\in \cF(M) | -\Delta f = f\}$ and recall that, by standard facts on compact K\"ahler-Einstein manifolds (see e.g. \cite{Be}, Ch. 2),
the  map $\imath: \cF(M) \longrightarrow \gX(M)$, $ \imath(f) = J (\n f)$, gives
 a $G$-equivariant isomorphism between $F$   and the Lie algebra  $\gg = \su_{n+1}$  of  Killing vector fields  of  $(M = \bC P^n, g_o)$.  Consider the maps
  $\psi_1,\psi_2:F \longrightarrow \cC^\infty(S^2 T^* M)$  defined by
 $$\psi_1(f) := fg_o\ ,\qquad  \psi_2(f) = \DD df$$
   (see also  \cite{Ko2}, p.659).
We  want to show  that  $\ker T =  \Im \psi_1   \oplus\Im \psi_2$ and  that
$$Z_{(g_o, f_o)}Ê= \ker T  \cap \ker \d =  \left \{ h = \DD df + \frac{1}{2} fg_o\ |\ f\in F\right\} \simeq F =  \su_{n+1}\ ,$$
from which (b) will immediately follow.  We first note that $\Im \psi_1\subset \ker T$ by definitions, while   $\Im \psi_2\subset \ker T$ because of the
following argument. Given a  local orthonormal frame field $\{e_i\}_{i=1,\ldots,n}$ with $\DD_{e_i}e_j|_{x_o}=0$ at a fixed point $x_o\in M$, we have at $x_o$
 $$\Delta (\DD df)_{ij} = \DD_l\DD_l\DD_i\DD_j f = \DD_l[\DD_i\DD_l\DD_jf + R_{lijp}\DD_pf] = $$
 $$= \DD_i\DD_l\DD_j\DD_lf + R_{lilp}\DD_p\DD_jf + R_{lijp}\DD_l\DD_pf + \DD_lR_{lijp}\DD_pf + R_{lijp}\DD_l\DD_pf = $$
$$= \DD_i[\DD_j\DD_l\DD_lf + R_{ljlp}\DD_pf] + 2 R_{lijp}\DD_l\DD_pf + \frac 12\DD_i\DD_jf =$$
$$= 2 R_{lijp}\DD_l\DD_pf = -2\cR(\DD df)_{ij},$$
where $R_{ijkl}=g(\R_{e_i,e_j}e_k,e_l)$.

Secondly,  if $h \in \Im \psi_1 \cap \Im \psi_2 $  (i.e.
$ h = \DD d f_1 = f_2 g_o$ for some $ f_1,f_2\in F$), since
\beq \label{4.7}(\delta h)_j = \DD_i\DD_i\DD_jf_1 = \DD_j\Delta f_1 + R_{ijip}\DD_pf_1 = -\DD_jf_1 + \frac{1}{2} \DD_jf_1 = -\frac{1}{2}
\DD_jf_1\ , \eeq
we have that
$$- d f_1 =  2 \delta h=
2 \delta(f_2 g_o) =  2 d f_2\quad \text{and}\quad  - d f_1 = d(\D f_1) = d(\Tr h) = 2n\, d f_2\ .$$
 If $n \geq 2$,  it follows that $d f_1 = d f_2 = 0$, so that    $f_1=f_2=0$ and $\Im \psi_1 \cap \Im \psi_2 = \{0\}$.
 From Lemma \ref{lemma1} and \eqref{4.7},  we see that  $\ker T =  \Im \psi_1 \oplus \Im \psi_2$ and that  all elements of  $ \ker T \cap \ker \d$ are of the form $h = \DD df + \frac{1}{2} fg_o$ for some $ f\in F$.\end{pf}
\begin{remark}  Theorem \ref{thm2} (b)   shows that  in a neighborhood of  the   standard metric  $g_o$ of $M = \bC P^n$, the moduli space $\Sol/\cD$  has {\it at most} dimension $n = \dim \su_{n+1}/\SU_{n+1} $. A  detailed study of the solitonic pre-moduli space at  the  Fubini-Study metric $g_o$ will be  the content  of  some of our  future investigations.
\end{remark}
\bigskip
The next proposition is another consequence   of  Proposition  \ref{Z=ZE} and   shows that  Einstein metrics with sufficiently large  positive curvatures are sol-rigid.
\begin{proposition} Let ($M,g$) be an $n$-dimensional Einstein metric with $\Ric = g$. If the sectional curvature $K$ is $\d$-pinched (i. e.  $\d \cdot K_{\max}\leq K\leq K_{\max}$ for some $\d \in (0, 1]$) such that
 \beq\label{K} K_{\min} \geq  \frac{1}{n}\qquad \text{and}\qquad  \d > \frac{n-2}{3n}\ , \eeq
 then ($M,g$) is sol-rigid  \end{proposition}
 \begin{proof} By \cite{Si}, the first condition implies that $2 \notin \mbox{Spec}(-\Delta,\cF(M))$, unless $(M,g)$ is isometric to a sphere. The condition on $\d$  implies that $(M,g)$ has no non trivial infinitesimal Einstein deformations (see \cite{Be}, Cor. 12.72, p. 357).  By Proposition \ref{Z=ZE}, the claim follows. \end{proof}

 \bigskip
 \section{Other  rigidity properties of Ricci solitons}
 \setcounter{equation}{0}
\begin{definition} An Einstein metric $g$ is said to be {\it weakly solitonic rigid\/} if there is a neighbourhood $\cU$ of $g$ in $\cM$ such that every Ricci soliton in $\mathcal U$ is Einstein.\end{definition}
The following proposition is a  consequence of a deep recent result in \cite{FLL}.
\begin{proposition} Let $g$ be an Einstein metric with Einstein constant $c > 0$  and   the diameter $d$. If
\beq \label{diam}  d\cdot \sqrt c < 2(\sqrt 2 - 1) \pi\ ,\eeq
then $g$ is weakly solitonic rigid.\end{proposition}
\begin{pf} In \cite{FLL} it is proved that a shrinking Ricci soliton $(g, f)$,  satisfying $\Ric + \DD df - cg = 0$ for some $c >0$,  is Einstein whenever
its diameter $d$ satisfies \eqref{diam}. If $D:\mathcal M\longrightarrow \bR$ is  the continuous map
$$ D(g') =  d(g')\cdot \left(\frac 1{n\operatorname{Vol}(g')}\,\int_M s_{g'}\, \mu_{g'}\right)^{1/2}\ ,$$
where $d(g')$ and $\operatorname{Vol}(g')$ are the diameter and the volume of $g'$, respectively,  then $\cU = D^{-1}((-\infty,2(\sqrt 2 - 1) \pi))$ works in the definition of weak rigidity.\end{pf}
\par
\medskip
 Consider now the case of a compact K\"abler manifold, that is a compact complex manifold $(M, J)$ admitting  a $J$-Hermitian metric $g$  with  closed K\"ahler form $\omega = g(\cdot, J \cdot)$. It is well known that, up to biholomorphisms,  there is   at most one K\"ahler Ricci soliton $g_o$ on $(M, J)$  (\cite{TZ}). On the other hand,   a recent result by Li (\cite{Li}) shows that,   for a given  K\"ahler Ricci soliton  $g_o$  on $(M, J)$ with   $G = \Iso^o(M, g_o)$,  if there exists  a smooth family   $\{J_t\}$  of $G$-invariant  (non biholomorphic) complex structures on $M$ with $J_0 = J$,  then there is also a family of $G$-invariant  Ricci solitons $g_t$  on $M$,  which are  K\"ahler w.r.t. the corresponding complex structures $J_t$.\par
 \par
 \smallskip
 The next proposition  is a generalization of a result  of Koiso (\cite{Ko1}, Thm. 10.5) and deals with weak solitonic rigidity of K\"ahler-Einstein metrics.
 \begin{proposition} Let $(M, J, g)$ be a $2n$-dimensional compact K\"ahler-Einstein manifold with $\Ric = g$. Assume that
$2$ is not an eigenvalue of the Hodge Laplacian $-\Delta$ acting on $\cF(M)$ and on the space of forms of type $(1,1)$. \par
If $H^2(M,\Theta) = 0$, where $\Theta$ is the sheaf of germs of holomorphic vector fields, then every Ricci soliton $g'$ on $M$, which  is sufficiently close to $g$,  is K\"ahler Einstein (w.r.t. a
suitable complex structure $J'$).
\end{proposition}
\begin{proof} By  \cite{Ko1}, Thm. 10.5, if $(M, J, g)$ is a  compact  Fano K\"ahler Einstein manifold,  any Einstein metric $g'$,  which is sufficiently close to $g$,  is K\"ahler with respect to some complex structure $J'$, provided the following conditions are satisfied:
\begin{itemize}
\item[a)] the  complex structure $J$ belongs to a non-singular complete family of complex structures,
\item[b)] there are no non-trivial holomorphic vector fields on $M$ and
\item[c)] any essential infinitesimal Einstein deformation,  which is $J$-Hermitian,  is necessarily trivial.
\end{itemize}
If this is the case,  denoting by $(g,f = const.)\in \cP^s$, $s\geq [n/2]+3$,   the  pair in $\cP^s$ corresponding to   $g$, the proof of \cite{Ko1}, Thm. 10.5, shows that
the set of  Einstein metric belonging to a sufficiently small Ebin's slice $\cS_{(g,f)} \subset \cM^s$  fills a real analytic submanifold $\cE^s \subset \cS_{(g,f)}$ with   $ E_{(g,f)} = T_{(g, f)} \cE^s$.\par
\medskip
We now observe that    $H^2(M,\Theta) = 0$ implies  (a), while    the assumption  $2\not\in \mbox{Spec}(-\Delta,\cF(M))$ implies (b) and the equality $Z_{(g,f)}Ê= E_{(g,f)}$, which are consequences of the Lichnerowicz Theorem (\cite{Be}, p. 90)  and Proposition \ref{Z=ZE}, respectively.  Moreover, if $h$  is an essential infinitesimal Einstein deformation (hence,   $\d h= 0 $ and $\Delta h + 2 \cR(h) = 0 $),   the   $(1,1)$-form $\psi := h\circ J$  is coclosed and is such that  $\Delta_H\psi = 2\psi$, where  $\Delta_H$ is the Hodge Laplacian of $(M, J, g)$ (see e.g. \cite{Be}, p. 362). Hence,  if $2\not\in \mbox{Spec}(-\Delta,\Omega^{1,1}(M))$,  also  (c) is satisfied and the quoted  result by  Koiso  applies.  \par
Moreover,
by Theorem \ref{thm2}, the real analytic submanifold $\cE^s$ is  contained in  $\cZ^s$ and,  since
$T_{(g,f)} \cZ^s = Z_{(g,f)}Ê= E_{(g,f)} = T_{(g,f)} \cE^s$, it follows  that $\cE^s \cap \cU= \cZ^s \cap \cU$ for any sufficiently small neighbourhood $\cU$ of $(g,f)$. This implies the claim.  \end{proof}
\begin{remark} The blow-up of the complex projective plane  at $\nu$ generic points,  with $\nu\geq 5$,  is  a compact K\"ahler Einstein manifold (\cite{Ti}) with  no non trivial holomorphic vector fields  and   $H^2(M,\Theta) = 0$ (see e.g.  \cite{Kod}). We do not know whether it also satisfies the above condition on the spectra of the Laplacian.\end{remark}

\bigskip\bigskip
\font\smallsmc = cmcsc8
\font\smalltt = cmtt8
\font\smallit = cmti8
\hbox{\parindent=0pt\parskip=0pt
\vbox{\baselineskip 9.5 pt \hsize=3.1truein
\obeylines
{\smallsmc
Fabio Podest\`a
Dip. di Matematica "U.Dini"
Universit\`a di Firenze
Viale Morgagni 67/A
I-50134 Firenze
ITALY}

\medskip
{\smallit E-mail}\/: {\smalltt podesta@math.unifi.it
}
}
\hskip 0.0truecm
\vbox{\baselineskip 9.5 pt \hsize=3.7truein
\obeylines
{\smallsmc
Andrea Spiro
Scuola di Scienze e Tecnologie
Universit\`a di Camerino
Via Madonna delle Carceri
I-62032 Camerino (Macerata)
ITALY
}\medskip
{\smallit E-mail}\/: {\smalltt andrea.spiro@unicam.it}
}
}

\end{document}